\documentclass[a4paper,10pt]{article}
\usepackage[utf8x]{inputenc}
\usepackage{tracefnt,amsmath,tabu,array}
\usepackage{amssymb,graphicx,setspace,amsfonts,amsbsy}
\usepackage{pifont,latexsym,ifthen,amsthm,rotating,calc,textcase,booktabs}

\newtheorem{theorem}{Theorem}[section]
\newtheorem{lemma}[theorem]{Lemma}
\newtheorem{corollary}[theorem]{Corollary}
\newtheorem{proposition}[theorem]{Proposition}
\newtheorem{definition}[theorem]{Definition}

\newtheorem{remark}[theorem]{Remark}
\newcommand{\filledbox}{\leavevmode
  \hbox to.77778em{%
  \hfil\vbox to.675em{\hrule width.6em height.6em}\hfil}}

\newcommand{\Rm}{{\mathbb R}}

\begin{document}
\tabulinesep=1.0mm
\title{Scattering of solutions to NLW by Inward Energy Decay\footnote{MSC classes: 35L71, 35L05; This work is supported by National Natural Science Foundation of China Programs 11601374, 11771325}}

\author{Ruipeng Shen\\
Centre for Applied Mathematics\\
Tianjin University\\
Tianjin, China}

\maketitle

\begin{abstract}
  The topic of this paper is a semi-linear, energy sub-critical, defocusing wave equation $\partial_t^2 u - \Delta u = - |u|^{p -1} u$ in the 3-dimensional space 
 ($3\leq p<5$) whose initial data are radial and come with a finite energy. In this work we prove scattering in the positive time direction by only assuming the inward part of the energy decays at a certain rate, as long as the total energy is finite, regardless of the decay rate or size of the outward energy. More precisely, we assume the initial data comes with a finite energy and 
\[
 \int_{\Rm^3} \max\{1,|x|^\kappa\}\left(\left|\nabla u_0(x)\cdot \frac{x}{|x|} + \frac{u_0(x)}{|x|} + u_1(x)\right|^2 + \frac{2}{p+1}|u_0(x)|^{p+1}\right) dx < \infty.
\] 
Here $\kappa\geq \kappa_0(p) = \frac{5-p}{p+1}$ is a constant. If $\kappa>\kappa_0(p)$, we can also prove $\|u\|_{L^p L^{2p}}(\Rm^+ \times \Rm^3)< +\infty$ and give an explicit rate of $u$'s convergence to a free wave.
\end{abstract}

\section{Introduction}

\subsection{Background}

In this work we consider the Cauchy problem of the 3D defocusing semi-linear wave equation 
\[
 \left\{\begin{array}{ll} \partial_t^2 u - \Delta u = - |u|^{p-1}u, & (x,t) \in \Rm^3 \times \Rm; \\
 u(\cdot, 0) = u_0; & \\
 u_t (\cdot,0) = u_1. & \end{array}\right.\quad (CP1)
\]
In this work we choose $p\in [3,5)$. We may consider initial data $(u_0,u_1)$ in two different basic spaces:
\begin{itemize}
 \item The critical Sobolev space $\dot{H}^{s_p}\times \dot{H}^{s_p-1}(\Rm^3)$ with $s_p = 3/2-2/(p-1)$. The importance of this space can be found if one considers the following rescaling: If $u$ is a solution to (CP1), then $u_\lambda (x,t) = \lambda^{-2/(p-1)} u(x/\lambda, t/\lambda)$ is also a solution with an identity 
 \[
  \|(u(\cdot,t_0),\partial_t u(\cdot,t_0)\|_{\dot{H}^{s_p}\times \dot{H}^{s_p-1}(\Rm^3)} 
  = \|(u_\lambda(\cdot,\lambda t_0),\partial_t u_\lambda (\cdot,\lambda t_0)\|_{\dot{H}^{s_p}\times \dot{H}^{s_p-1}(\Rm^3)} 
 \] 
If $p=5$, then the critical Sobolev space is exactly the energy space $\dot{H}^1\times L^2$. Thus $p=5$ is called the energy critical case. In this case the asymptotic behaviour of solution to defocusing equation has been completely understood. Any given solution always resembles a solution to the linear wave equation as $t\rightarrow+\infty$ (i.e. the solution scatters) as long as initial data have a finite $\dot{H}^1 \times L^2$ norm. This result was proved by M. Grillakis \cite{mg1} in 1990's. In this work we focus on the energy-subcritical case. In particular we choose $3\leq p<5$. It is conjectured that any solution with initial data in the critical Sobolev space must scatter. However, nobody can prove this result without additional assumptions, as far as the author knows. Please see \cite{radialcubic, cubic3dwave, shen2} for some already known results regrading global behaviour of solutions with initial data in the critical Sobolev spaces.
 \item The energy space $(\dot{H}^1\cap L^{p+1}) \times L^2 (\Rm^3)$. The advantage of this space is its strong physical meaning. In the energy critical case $p=5$ this happens to be the critical Sobolev space. Solutions with initial data in this space satisfies an energy conservation law:
\[
 E(u, u_t) = \int_{\Rm^3} \left(\frac{1}{2}|\nabla u(\cdot, t)|^2 +\frac{1}{2}|u_t(\cdot, t)|^2 + \frac{1}{p+1}|u(\cdot,t)|^{p+1}\right)\,dx = \hbox{Const}.
\] 
A combination of a fixed-point argument, Strichartz estimates and the energy conservation law gives the global existence of the solution. This global existence of solutions has been known for many years. However, the asymptotic behaviour of solution, i.e. whether the solution always resembles a free wave as $t\rightarrow \infty$, is still unknown, as far as the author knows.
\end{itemize}
The following are some of the known results about the scattering theory of solutions with a finite energy. All of these make additional assumptions on the initial data besides the finiteness of the energy.

\paragraph{Conformal conservation law} The following conformal conservation law was known many years ago.
\[
  \frac{d}{dt} Q(t, u, u_t) = \frac{4(3-p)t}{p+1} \int_{\Rm^3} \left|u(x,t)\right|^{p+1} dx \leq 0.
\]
Here $Q(t,\varphi,\psi) = Q_0(t, \varphi, \psi) + Q_1(t, \varphi)$ is called the conformal charge with
\begin{align*}
 Q_0(t,\varphi,\psi) &= \left\|x\psi + t \nabla \varphi \right\|_{L^2(\Rm^3)}^2 + \left\|(t\psi+2\varphi)\frac{x}{|x|} +|x|\nabla \varphi\right\|_{L^2(\Rm^3)}^2;\\
 Q_1(t, \varphi) &= \frac{2}{p+1}\int_{\Rm^3} (|x|^2+t^2)|\varphi(x,t)|^{p+1} dx.
\end{align*}
This gives us a pretty fast decay rate of $t^{-2}$ for $\int_{\Rm^3} |u(x,t)|dx$ as $t\rightarrow \infty$ thus scattering of the solution as long as $Q(0,u_0,u_1)<\infty$. Roughly speaking, we need to assume
\begin{equation} \label{condition1}
    \int_{\Rm^3} \left[(|x|^2+1) (|\nabla u_0 (x)|^2 + |u_1(x)|^2) + |u_0(x)|^2 \right] dx < \infty.
\end{equation}
More details can be found in \cite{conformal2,conformal}.
\paragraph{Energy Distribution of Radial Solution} The author considered radial solution with a finite energy whose decay rate is much lower as $x\rightarrow \infty$ in recent work \cite{energy dis} and gave the following two theorems. The first one is a scattering theory and also an application of the second one. The second one gives information about energy distribution in the space-time as well as some asymptotic behaviour for all radial solutions to (CP1) with a finite energy.
\begin{theorem} \label{scattering with full energy decay}
Assume $3\leq p<5$. Let $\kappa>\kappa_0(p)=\frac{5-p}{p+1}$ be a constant. If $u$ is a radial solution to (CP1) with initial data $(u_0,u_1)$ so that
 \[
  \int_{\Rm^3} (|x|^\kappa+1)\left(\frac{1}{2}|\nabla u_0|^2 + \frac{1}{2}|u_1|^2+\frac{1}{p+1}|u_0|^{p+1}\right) dx < +\infty,
 \]
 then $u$ must scatter in the space $\dot{H}^1\times L^2$ in both two time directions.
\end{theorem}
\begin{theorem} \label{energy distribution}
Assume $3\leq p<5$. Let $u$ be a radial solution to (CP1) with a finite energy $E$. Then there exist a three-dimensional free wave $v^-(x,t)$, with an energy $\tilde{E}_-\leq E$, so that 
 \begin{itemize}
  \item We have scattering outside any backward light cone ($R \in \Rm$)
   \[
     \lim_{t \rightarrow - \infty} \left\|\left(\nabla u(\cdot,t), u_t(\cdot,t)\right)- \left(\nabla v^- (\cdot,t), v_t^- (\cdot,t)\right)\right\|_{L^2(\{x\in \Rm^3:|x|>R+|t|\})} = 0.
   \]
   \item If we have $\tilde{E}_- = E$, then the scattering happens in the whole space in the negative time direction.
    \[
     \lim_{t \rightarrow - \infty} \left\|\left(u(\cdot,t), u_t(\cdot,t)\right)- \left(v^- (\cdot,t), v_t^- (\cdot,t)\right)\right\|_{\dot{H}^1\times L^2(\Rm^3)} = 0.
   \]
   \item If $\tilde{E}_- < E$, the remaining energy (also called ``retarded energy'') can be located: for any constants $c \in (0,1)$ and $\beta<\frac{2(p-2)}{p+1}$ we have
   \[
    \lim_{t \rightarrow - \infty} \int_{c|t|<|x|<|t|-|t|^\beta} \left(\frac{1}{2}|\nabla u(x,t)|^2 + \frac{1}{2}|u_t(x,t)|^2 + \frac{1}{p+1}|u(x,t)|^{p+1}\right) dx = E - \tilde{E}_-.
   \]
 \end{itemize}
 The asymptotic behaviour in the positive time direction is similar.
\end{theorem}
The idea of Theorem \ref{energy distribution} is the following: We first split the solution into inward wave and outward wave, thus split the energy into inward and outward energies. Then we prove an energy flux formula for inward/outward energy. Finally we are able to consider and understand the distribution property of energies in the space-time as well as the rules of interaction and transformation of inward and outward energies. 
\subsection{Main Results}

In this paper we always consider radial solutions to (CP1) with a finite energy. We combine the idea of inward/outward energies introduced in \cite{energy dis} as mentioned above with a weighted Morawetz estimate to obtain the following new results. 
\begin{itemize}
 \item We prove the scattering of solution in the positive time direction by assuming the inward energy of initial data decays at the same rate as in  the previous work, regardless of the decay rate or size of outward energy. The idea comes from the author's intuition: 
 \begin{itemize}
   \item [(1)] Outward wave moves away from the origin and thus has a tendency to scatter. It seems that the size or decay rate of outward energy will not be an obstacle to scattering.
   \item [(2)] In the defocusing case, inward energy may be transformed into outward energy in two different ways. First, inward wave carrying energy may move through the origin and become outward wave by linear propagation. Second, nonlinear effect transforms inward energy into outward energy everywhere every time. However, outward energy can never be transformed into inward energy.
 \end{itemize}
 \item We can also prove the scattering in the endpoint case of the decay rate $\kappa = \kappa_0(p)$.
 \item If $\kappa>\kappa_0(p)$, we prove that $\|u\|_{L^p L^{2p}}(\Rm^+\times \Rm^3) < +\infty$. This is equivalent to saying that the nonlinear part satisfies $\|-|u|^{p-1}u\|_{L^1 L^2(\Rm^+\times \Rm^3)}<+\infty$. In fact we obtain
 \[
  \left\|-|u|^{p-1} u\right\|_{L^1 L^2 ([t_0,\infty)\times \Rm^3)} \leq C t_0^{-\frac{p+1}{p+3}(\kappa-\kappa_0(p))}, \; t_0 \geq 2.
 \]
 According to the classic Strichartz estimate, this not only gives the scattering in the space $\dot{H}^1 \times L^2$ but also gives an explicit convergence rate. Because we ignore the sign of $u$ when we consider the $L^p L^{2p}$ norm, we claim that in this case the solution $u$ scatters``absolutely''. 
\end{itemize}
Now we give the precise statement of our main theorem. 

\begin{theorem} \label{main 1}
Fix $3\leq p<5$ and $\kappa_0 (p) = \frac{5-p}{p+1}$. Assume that $\kappa \in [\kappa_0(p),1)$ is a constant. Let $u$ be a radial solution to (CP1) with a finite energy $E$ so that 
\[
 K = \int_{\Rm^3} \max\{1,|x|^\kappa\}\left(\left|\nabla u_0(x)\cdot \frac{x}{|x|} + \frac{u_0(x)}{|x|} + u_1(x)\right|^2 + \frac{2}{p+1}|u_0(x)|^{p+1}\right) dx < \infty.
\]
Then we have 
\begin{itemize}
 \item [(a)] The solution $u$ scatters in the positive time direction. More precisely, there exists $(v_0 ,v_1) \in \dot{H}^1 \times L^2(\Rm^3)$, so that 
 \[
  \lim_{t \rightarrow + \infty} \left\|\begin{pmatrix} u(\cdot,t)\\ \partial_t u(\cdot,t)\end{pmatrix} - 
  \mathbf{S}_L (t)\begin{pmatrix}v_0 \\ v_1\end{pmatrix}\right\|_{\dot{H}^1 \times L^2(\Rm^3)} = 0.
 \]
 Here $\mathbf{S}_L (t)$ is the linear wave propagation operator.\\
 \item [(b)] If $\kappa > k_0(p)$, then the solution $u$ has a finite norm $\|u\|_{L^p L^{2p}(\Rm^+ \times \Rm^3)} < +\infty$.  In fact we have the following estimates for $t \geq 2$.
 \begin{align*}
  \|u\|_{L^p L^{2p} ([t,\infty)\times \Rm^3)}^p &\leq C t^{-\frac{p+1}{p+3}(\kappa-\kappa_0(p))}. \\
  \left\|\begin{pmatrix} u(\cdot,t)\\ u_t(\cdot,t)\end{pmatrix} - \mathbf{S}_L (t) \begin{pmatrix} v_0\\ v_1\end{pmatrix}\right\|_{\dot{H}^1\times L^2}
  & \leq C t^{-\frac{p+1}{p+3}(\kappa-\kappa_0(p))}.
 \end{align*}
 Here $C$ is a constant independent of the time $t$.
 \end{itemize}
\end{theorem}
\begin{remark}
 Since wave equation is time-reversible, if the outward energy decays so that
\[
  \int_{\Rm^3} \max\{1,|x|^\kappa\}\left(\left|\nabla u_0(x)\cdot \frac{x}{|x|} + \frac{u_0(x)}{|x|} - u_1(x)\right|^2 + \frac{2}{p+1}|u_0(x)|^{p+1}\right) dx < \infty,
\]
then the solution must scatter in the negative time direction. 
\end{remark}
\begin{remark}
If we solve (CP1) in the critical Sobolev space $\dot{H}^{s_p} \times \dot{H}^{s_p-1}(\Rm^3)$, then it has been known that the scattering of a solution $u$ in this critical space is equivalent to the finiteness of the norm $\|u\|_{L^{2(p-1)} L^{2(p-1)}([0,T_+)\times \Rm^3)} < +\infty$. Here $[0,T_+)$ is the maximal lifespan of solution $u$ in the positive time direction. In this work we solve (CP1) in the energy space and consider the $L^p L^{2p}$ norm instead. In fact, we give an example of solution $u$ to (CP1) in Appendix, so that $u$ scatters in $\dot{H}^1 \times L^2$, but $\|u\|_{L^{2(p-1)}L^{2(p-1)}(\Rm^+ \times \Rm^3)} = + \infty$. This gives a phenomenon of scattering which can never be covered by the scattering theory in the critical Sobolev space. 
\end{remark}

\subsection{The idea}

In this subsection we give the main idea and sketch the proof of the main theorem. Details can be found in later sections. First of all, we follow the standard procedure to first convert the 3-dimensional wave equation to a one-dimensional one. This helps us to take full advantage of the radial assumption.
 
\paragraph{Transformation to 1D} Let $u$ be a radial solution to (CP1). If we define $w(r,t) = r u(x,t)$, where $|x|=r$, then $w$ satisfies the following one-dimensional wave equation
\[
 w_{tt} - w_{rr} = - \frac{|w|^{p-1}w}{r^{p-1}}.
\]
An integration by parts shows that
\begin{align}
 &2\pi \int_a^b (|w_r(r,t)|^2+|w_t(r,t)|^2) dr \nonumber \\
 &\qquad = 2\pi\left[\int_a^b \left(r^2|u_r(r,t)|^2 + r^2|u_t(r,t)|^2\right)dr +b|u(b,t)|^2 - a|u(a,t)|^2\right]. \label{energy transformation}
\end{align}
Because we have the following limits for any radial $\dot{H}^1(\Rm^3)$ function $f$,
\begin{equation*}
 \lim_{r\rightarrow 0^+} r|f(r)|^2  = \lim_{r\rightarrow +\infty} r|f(r)|^2 = 0.
\end{equation*}
The new solution $w$ has the same kinetic energy as $u$ at any given time. 
\[
 2\pi \int_0^\infty (|w_r(r,t)|^2+|w_t(r,t)|^2) dr = \int_{\Rm^3} \left(\frac{1}{2}|\nabla u(x,t)|^2 + \frac{1}{2}|u_t(x,t)|^2 \right) dx.
\]
A simple calculation shows they also share the same potential energy. In summary, we have
\begin{align*}
 E(w,w_t) \doteq 2\pi \int_0^\infty (|w_r(r,t)|^2+|w_t(r,t)|^2+\frac{2}{p+1}\cdot \frac{|w(r,t)|^{p+1}}{r^{r-1}}) dr = E(u,u_t) = \hbox{Const}.
\end{align*}

\paragraph{Weighted Morawetz} Next we unitize the already-known energy distribution method, but with an additional weight, to obtain the following weighted Morawetz estimates, as long as the inward energy of initial data decays in a certain rate
\begin{align*}
 &\int_0^\infty t^{\kappa}|\xi(t)|^2 dt \leq C K,& &\int_0^\infty \int_0^\infty \frac{(r+t)^\kappa |w(r,t)|^{p+1}}{r^p} drdt \leq C K.&
\end{align*}
Here $\xi$ is an $L^2$ function found in the energy flux formula, see Proposition \ref{energy flux formula}. Given a time interval $(t_1,t_2)$, the value of $\pi \int_{t_1}^{t_2} |\xi(t)|^2 dt$ is exactly the amount of energy transformed from inward energy to outward energy at the origin due to linear propagation during the given time period. $K$ is defined in the main theorem and determined solely by inward wave part of initial data. The letter $C$ represents a constant. We can also rewrite the weighted Morawetz in term of an integral about $u$
\[
 \int_0^\infty \int_{\Rm^3} \frac{(|x|+t)^\kappa |u(x,t)|^{p+1}}{|x|} dx dt\leq C K.
\]
The author would like to mention that the integral of $|u(x,t)|^{p+1}/|x|$ over a region in $\Rm_t\times \Rm_x^3$ actually measures the amount of energy transformed from inward energy to outward energy by nonlinear effect in the given region.

\paragraph{Decay of Energies} The first application of the weighted Morawetz estimates is an explicit decay rate of the total inward energy at time $t$ 
\[
  E_-(t) \leq CK t^{-\kappa}, \quad t>0.
 \]
In addition, if $t>r>0$, we can also give an upper bound of the outward energy in the ball $B(0,r)$ at the time $t$
\[
 E_+(t;0,r) \leq \frac{CKt^{1-\kappa}}{t-r}.
\]
Moreover, we also have an integral estimate on the energy inside a ball of radius $R$. 
\begin{equation} \label{integral energy estimate}
 \int_{t_0}^\infty E(t;0,R) dt \leq C KRt_0^{-\kappa}, \quad t_0>R>0. 
\end{equation}

\paragraph{Proof of Scattering} If $\kappa \geq \kappa_0(p)$, a combination of these upper bounds on energies with the energy distribution property given in Theorem \ref{energy distribution} leads to a contradiction, unless the solution scatters in the positive time direction. This proves part (a) of our main theorem. 

\paragraph{Upper bound on $L^p L^{2p}$ norm} According to the local theory, the $L^p L^{2p}$ norm in a small time interval is always finite. Thus it suffices to  verify the decay rate of $\|u\|_{L^p L^{2p}([t_0,\infty)\times \Rm^3)}$. We split the solution $u$ into three parts 
\[
 u = u\cdot \chi_{|x|>t/4} + u\cdot \chi_{1/4<|x|<t/4} + u\cdot \chi_{|x|<1/4}.
\]
Here $\chi$'s are characteristic functions of the regions indicated. We discuss these three cases one by one.
\begin{itemize}
 \item Large radius case $|x|>t/4$. This is the dominant term, since most energy concentrates around the light cone $|x|=t$. We can find an upper bound of this term by 
 \[
  \|u\chi_{|x|>t/4}\|_{L^p L^{2p}([t_0,\infty)\times \Rm^3)}^p = \int_{t_0}^\infty \left((\sup_{|x|>t/4} |u(x,t)|)^{p-1} \int_{|x|>t/4} |u|^{p+1} dx\right)^{1/2} dt
 \]
Here the integral $\int |u|^{p+1} dx$ is part of the inward energy, thus decays like $t^{-\kappa}$. An upper bound of $\sup u(x,t)$ can be found by the following pointwise estimate given in Lemma \ref{local upper bound pointwise}.
\begin{align*}
  |u(x,t)| &\leq C |x|^{\frac{-4}{p+3}}\left(\int_0^{|x|} |w_r(r,t)|^2 dr\right)^{\frac{1}{p+3}} \left(\int_0^{|x|} \frac{|w(r,t)|^{p+1}}{r^{p-1}} dr\right)^{\frac{1}{p+3}}\\
  &\leq C |x|^{\frac{-4}{p+3}} E(t;0,|x|)^\frac{1}{p+3} E_-(t;0,|x|)^\frac{1}{p+3}.
\end{align*}
 \item Medium radius case $1/4<|x|<t/4$. We can further write the region $\{(x,t): t>t_0, 1/4<|x|<t/4\}$ as a big union of subsets
 \[
  \left\{(x,t): t>t_0, \frac{1}{4}<|x|<\frac{t}{4}\right\} = \bigcup_{j=1}^n \bigcup_{i=0}^\infty \left\{(x,t): \frac{t^{\alpha_{j-1}}}{4}<|x|<\frac{t^{\alpha_j}}{4}, 2^i t_0 < t < 2^{i+1} t_0\right\}.
 \]
Here $0=\alpha_0< \alpha_1<\alpha_2 < \cdots < \alpha_n=1$ are determined by $p$ and $\kappa$ but nothing else. The restriction of $u$ on each subset can be dealt with by a combination of the same pointwise estimate given above and the integral estimate \eqref{integral energy estimate}. We then combine the estimates on these subsets together in a suitable way to verify the decay rate.
 \item Small radius case $|x|<1/4$. In this case we first apply the local theory to show that there exists a small time $T$, so that for any $t>2$ we have 
 \[
  \|u\|_{L^p L^{2p}([t-T,t+T]\times B(0,1/4))} \leq C E(t;0,1)^{1/2}.
 \]
 We observe that $E(t;0,1)$ in the right hand side is not only uniformly bounded, but also satisfies an integral estimate \eqref{integral energy estimate}. As a result, we can break the time interval $[t_0,\infty)$ into small sub-intervals, then combine the local $L^p L^{2p}$ estimates given above, and finally obtain a global-in-time $L^p L^{2p}$ estimate of $u\chi_{|x|<1/4}$.
\end{itemize}

\subsection{The Structure of This Paper}
 
This paper is organized as follows. In section 2 we collect notations, give a useful pointwise estimate and recall a few preliminary results about the inward/outward energy. Next in Section 3 we prove a weighted Morawetz estimate and give decay estimates on the inward/outward energy. Then we prove the main theorem in section 4. The appendix is devoted to a specific example of scattering solution, which can not be covered by the classic scattering theory in the critical Soblev space.
 
\section{Preliminary Results} \label{sec:pre}

\subsection{Notations}

\paragraph{The $\lesssim$ symbol} We use the notation $A \lesssim B$ if there exists a constant $c$, so that the inequality $A \leq c B$ always holds.  In addition, a subscript of the symbol $\lesssim$ indicates that the constant $c$ is determined by the parameter(s) mentioned in the subscript but nothing else. In particular, $\lesssim_1$ means that the constant $c$ is an absolute constant. 

\paragraph{Radial functions} Let $u(x,t)$ be a spatially radial function. By convention $u(r,t)$ represents the value of $u(x,t)$ when $|x| = r$. 

\paragraph{Energies and flux} we continues to use the notation from the author's recent work \cite{energy dis}. The values of all the energies and fluxes below can never exceed the total energy. 
\begin{definition} \label{energies}
Let $u$ be a radial solution to (CP1) and $w=ru$. We define the inward and outward energies
\begin{align*}
E_- (t) & = \pi \int_{0}^{\infty} \left(\left|w_r(r,t)+w_t(r,t)\right|^2+\frac{2}{p+1}\cdot\frac{|w(r,t)|^{p+1}}{r^{p-1}}\right) dr\\
 E_+ (t) & = \pi \int_{0}^{\infty} \left(\left|w_r(r,t)-w_t(r,t)\right|^2+\frac{2}{p+1}\cdot\frac{|w(r,t)|^{p+1}}{r^{p-1}}\right) dr
\end{align*}
We can also rewrite these energies in term of $u$
\[
 E_\mp (t) = \int_{\Rm^3} \left(\frac{1}{4}\left|\nabla u(x,t)\cdot \frac{x}{|x|} + \frac{u(x,t)}{|x|} \pm u_t(x,t)\right|^2 + \frac{1}{2(p+1)} |u(x,t)|^{p+1}\right) dx.
\]
We also need to consider their truncated versions
\begin{align*}
 E_- (t;r_1,r_2) & = \pi \int_{r_1}^{r_2} \left(\left|w_r(r,t)+w_t(r,t)\right|^2+\frac{2}{p+1}\cdot\frac{|w(r,t)|^{p+1}}{r^{p-1}}\right) dr\\
 E_+ (t;r_1,r_2) & = \pi \int_{r_1}^{r_2} \left(\left|w_r(r,t)-w_t(r,t)\right|^2+\frac{2}{p+1}\cdot\frac{|w(r,t)|^{p+1}}{r^{p-1}}\right) dr\\
 E(t;r_1,r_2) & = E_-(t;r_1,r_2) + E_+(t;r_1,r_2)
 \end{align*}
\end{definition}
\noindent The energy fluxes defined below are the amount of inward energy that moves across a characteristic line $t+r=s$ and the amount of outward energy  that moves across a characteristic line $t-r = \tau$, respectively. These half-lines ($r>0$) correspond to backward/forward light cones centred at the $t$-axis in the original space time $\Rm_x^3\times \Rm_t$.
\begin{definition}[Notations of Energy fluxes]
 Given $s,\tau \in \Rm$, we define full energy fluxes
\begin{align*} 
 Q_-^- (s) & =  \frac{4\pi}{p+1} \int_{-\infty}^{s} \frac{|w(s-t,t)|^{p+1}}{(s-t)^{p-1}} dt;\\
 Q_+^+(\tau) & = \frac{4\pi}{p+1} \int_{\tau}^{+\infty} \frac{|w(t-\tau,t)|^{p+1}}{(t-\tau)^{p-1}} dt.
\end{align*}
We can also consider their truncated version, which is the energy flux across a segment of a characteristic line.  
\begin{align*} 
 Q_-^- (s; t_1,t_2) & =  \frac{4\pi}{p+1} \int_{t_1}^{t_2} \frac{|w(s-t,t)|^{p+1}}{(s-t)^{p-1}} dt,& &t_1<t_2\leq s;\\
 Q_+^+(\tau; t_1,t_2) & = \frac{4\pi}{p+1} \int_{t_1}^{t_2} \frac{|w(t-\tau,t)|^{p+1}}{(t-\tau)^{p-1}} dt,& &\tau\leq t_1<t_2.
\end{align*}
\end{definition}


\subsection{Uniform Pointwise Estimates} \label{sec:pointwise estimate}
In this subsection we first recall
\begin{lemma}[Please see Lemma 3.2 of Kenig and Merle's work \cite{km}] 
If $u$ is a radial $\dot{H}^1(\Rm^3)$ function, then 
\[
 |u(r)| \lesssim_1 r^{-1/2} \|u\|_{\dot{H}^1(\Rm^3)}.
\]
\end{lemma}
\noindent Therefore we always have $|w(r,t)| \lesssim_1 E^{1/2} r^{1/2}$. We also need to use the following refined pointwise estimate. This utilizes the information from local $\dot{H}^1$ and $L^{p+1}$ norms of the solution. 
\begin{lemma} \label{local upper bound pointwise}
 Let $u$ be a radial $\dot{H}^1(\Rm^3)$ function and $w(r,t)=ru(r,t)$. The we have
 \[
  |w(R,t)| \lesssim_p R^{\frac{p-1}{p+3}}\left(\int_0^R |w_r(r,t)|^2 dr\right)^{\frac{1}{p+3}} \left(\int_0^R \frac{|w(r,t)|^{p+1}}{r^{p-1}} dr\right)^{\frac{1}{p+3}} 
 \]
\end{lemma}
\begin{proof}
For convenience let us omit the time $t$ in the argument below and use the notations 
\begin{align*}
 &E_1 = \int_0^R |w_r(r)|^2 dr, & &E_2 = \int_0^R \frac{|w(r)|^{p+1}}{r^{p-1}} dr.&
\end{align*}
First of all, we observe ($0< r< R$)
 \begin{equation}
  |w(R)-w(r)| = \left|\int_{r}^R w_r(s) ds \right| \leq |R-r|^{1/2} \left(\int_{r}^R |w_r(s)|^2 ds\right)^{1/2} \leq (E_1)^{1/2} |R-r|^{1/2}. \label{continuity half}
 \end{equation}
 We recall $w(r)\rightarrow 0$ as $r\rightarrow 0^+$ thus $|w(R)|\leq R^{1/2}E_1^{1/2}$. Let us consider $r \in [R-{|w(R)|^2}/{4E_1},R]\subseteq [3R/4,R]$. The inequality \eqref{continuity half} implies $|w(r)-w(R)|<|w(R)|/2$ thus $|w(r)|>|w(R)|/2$.  As a result
\begin{align*}
E_2 \geq \int_{R-{|w(R)|^2}/{4E_1}}^{R} \frac{|w(r)|^{p+1}}{r^{p-1}} dr \geq \frac{(|w(R)|/2)^{p+1}}{R^{p-1}}\cdot |w(R)|^2/4E_1 \gtrsim_p \frac{|w(R)|^{p+3}}{R^{p-1}E_1}.
\end{align*} 
Thus we have $w(R) \lesssim_p R^{(p-1)/(p+3)} E_1^{1/(p+3)} E_2^{1/(p+3)}$ and finish the proof.
\end{proof}

\subsection{Energy Flux Formula and Energy Distribution}

In this subsection we recall some results regarding energy distribution and transformation from the author's recent work \cite{energy dis}.

\begin{proposition}[General Energy Flux, see Proposition 3.1 and Remark 3.3 of \cite{energy dis}] \label{energy flux formula}
 Let $\Omega$ be a closed region in the right half $(0,\infty)\times \Rm$ of $r-t$ space. Its boundary $\Gamma$ consists of finite line segments, which are paralleled to either $t$-axis, $r$-axis or characteristic lines $t\pm r = 0$, and is oriented counterclockwise. Then we have
 \begin{align}
  \pi \int_{\Gamma} \left(|w_r+w_t|^2 + \frac{2}{p+1}\cdot \frac{|w|^{p+1}}{r^{p-1}}\right) dr & + \left(|w_r+w_t|^2 - \frac{2}{p+1}\cdot \frac{|w|^{p+1}}{r^{p-1}}\right) dt \nonumber \\
  & \qquad -\frac{2\pi(p-1)}{p+1} \iint_{\Omega} \frac{|w|^{p+1}}{r^p} dr dt = 0. \label{energy flux inward}\\
  \pi \int_{\Gamma} \left(|w_r-w_t|^2 + \frac{2}{p+1}\cdot \frac{|w|^{p+1}}{r^{p-1}}\right) dr & + \left(-|w_r-w_t|^2 + \frac{2}{p+1}\cdot \frac{|w|^{p+1}}{r^{p-1}}\right) dt \nonumber \\
  &\qquad  + \frac{2\pi(p-1)}{p+1} \iint_{\Omega} \frac{|w|^{p+1}}{r^p} dr dt = 0 \label{energy flux outward}. 
 \end{align}
 Furthermore, there exists a function $\xi \in L^2(\Rm)$ with $\|\xi\|_{L^2}^2 \lesssim_p E$, which is solely determined by $u$ and independent of $\Omega$, so that the identities above also hold for regions $\Omega$ with part of their boundary on the $t$-axis. In this case the line integral from the point $(0,t_2)$ downward to $(0,t_1)$ along $t$-axis is equal to either $- \pi \int_{t_1}^{t_2} |\xi(t)|^2 dt$ in identity \eqref{energy flux inward} or $\pi \int_{t_1}^{t_2} |\xi(t)|^2 dt$ in identity \eqref{energy flux outward}.
\end{proposition}
\begin{remark}
 According to the following classic Morawetz estimate, the double integral in the energy flux formula above is always finite
 \[
  \int_{-\infty}^{\infty} \int_{\Rm^3} \frac{|u(x,t)|^{p+1}}{|x|} dx dt \lesssim_p E \quad \Rightarrow \quad \iint_{\Omega} \frac{|w(r,t)|^{p+1}}{r^p} dr dt \lesssim_p E.
 \]
\end{remark}
We apply the energy flux formula on the triangles $\Omega_u=\{(r,t): r>0, t>t_0, r+t<r_0+t_0\}$ and $\Omega_d =\{(r,t): r>0, t<t_0, t-r>t_0-r_0\}$ to obtain the following triangle law. This will be frequently used in this work.
\begin{proposition}[Triangle Law] \label{triangle law}
Let $t_0\in \Rm$ and $r_0>0$
\begin{align*}
 E_-(t_0;0,r_0) = & \pi \int_{t_0}^{t_0+r_0} \!\!|\xi(t)|^2 dt + Q_-^-(r_0+t_0;t_0,r_0+t_0)+ \frac{2\pi(p-1)}{p+1} \iint_{\Omega_{u}}\! \frac{|w(r,t)|^{p+1}}{{r}^p} dr dt;\\
 E_+(t_0;0,r_0) = & \pi \int_{t_0-r_0}^{t_0} \!|\xi(t)|^2 dt + Q_+^+(t_0-r_0;t_0-r_0,t_0)+ \frac{2\pi(p-1)}{p+1} \iint_{\Omega_{d}}\! \frac{|w(r,t)|^{p+1}}{{r}^p} dr dt.
\end{align*}
\end{proposition}
\noindent The following result can be understood as a triangle law for $E_-$ with $r_0=+\infty$. 
\begin{proposition} \label{infinite triangle}
Given any time $t$,
\begin{equation*}
 E_-(t) = \pi \int_{t}^{\infty} |\xi(t')|^2 dt' + \frac{2\pi (p-1)}{p+1} \int_{t}^{+\infty} \!\!\int_0^{+\infty} \frac{|w(r,t')|^{p+1}}{r^p} dr dt'. 
\end{equation*}
\end{proposition}
\noindent The following propositions give some known asymptotic behaviour of energies and fluxes.
\begin{proposition}[Monotonicity of Inward and Outward Energies] \label{monotonicity of energies}
 The inward energy $E_-(t)$ is a decreasing function of $t$, while the outward energy $E_+(t)$ is an increasing function of $t$. In addition 
\begin{align*}
 &\lim_{t\rightarrow +\infty} E_-(t) = 0; & &\lim_{t\rightarrow -\infty} E_+(t) = 0.&
\end{align*}
\end{proposition}
\begin{proposition} \label{a limit of Q}
 We have the limits
 \begin{align*}
  &\lim_{s \rightarrow +\infty} Q_-^-(s) = 0;& &\lim_{\tau \rightarrow -\infty} Q_+^+(\tau) = 0.&
 \end{align*}
\end{proposition}

The following proposition gives convergence of $w_r\pm w_t$ along corresponding characteristic lines from the energy point of view. We will use it to deal with the endpoint $\kappa = \kappa_0(p)$.
\begin{proposition} \label{limits of wr pm wt}
There exist functions $g_-(s), g_+(\tau)$ with $\|g_-\|_{L^2(\Rm)}^2, \|g_+\|_{L^2(\Rm)}^2 \leq E/\pi$. so that 
\begin{align*}
 &\left|(w_r+w_t)(s-t,t) - g_-(s)\right| \lesssim_{p,E} (s-t)^{-\frac{p-2}{p+1}},& &t<s;&\\
 &\left|(w_r-w_t)(t-\tau,t) - g_+(\tau)\right| \lesssim_{p,E} (t-\tau)^{-\frac{p-2}{p+1}},& &t>\tau.&
\end{align*}
Furthermore we have the following $L^2$ convergence for any given $s_0,\tau_0 \in \Rm$
\begin{align*}
 &\lim_{t \rightarrow -\infty} \|(w_r+w_t)(s-t,t) - g_-(s)\|_{L_s^2([s_0,\infty))} = 0; \\
 &\lim_{t \rightarrow +\infty} \|(w_r-w_t)(t-\tau,t) - g_+(\tau)\|_{L_\tau^2((-\infty, \tau_0])} = 0.
\end{align*}
\end{proposition}

\section{Weighted Morawetz}

First of all, the author would like mention that all propositions and decay estimates below still hold even if $\kappa < \kappa_0(p)$, as long as $u$ and $\kappa>0$ satisfy all other conditions in the main theorem. We start by proving a few weighted Morawetz estimates under the assumption that the inward energy decays at a certain rate. 
  
\begin{proposition}[General Weighted Morawetz] \label{Weighted Morawetz Inequality}
 Let $0<\gamma<1$ and $a \in C^1([1,\infty))$ be a positive function satisfying 
 \begin{align*}
  &a(1)=1,& &0<a'(r)\leq \gamma a(r)/r.&
 \end{align*}
 Assume that $u$ is a solution to (CP1) with a finite energy so that $w(r,t)=ru(r,t)$ satisfies 
 \begin{align*}
  K_1 = & \pi \int_0^1 \left(\left|w_r(r,0)+w_t(r,0)\right|^2+\frac{2}{p+1}\cdot\frac{|w(r,0)|^{p+1}}{r^{p-1}}\right) dr \\
  & \quad +\pi \int_1^\infty a(r) \left(\left|w_r(r,0)+w_t(r,0)\right|^2+\frac{2}{p+1}\cdot\frac{|w(r,0)|^{p+1}}{r^{p-1}}\right) dr < \infty,
 \end{align*}
 then we have 
 \begin{align*}
  &\pi \int_1^\infty a(t)|\xi(t)|^2 dt \leq K_1,& &\frac{2\pi(p-1)}{p+1}\iint_{\Omega_1} \frac{a(r+t)|w(r,t)|^{p+1}}{r^p} drdt \leq \frac{p-1}{p-1-2\gamma} K_1,& 
 \end{align*}
 where $\xi$ is the $L^2$ function in the energy flux formula and $\Omega_1 = \{(r,t):r+t>1,t>0,r>0\}$.
\end{proposition}

 \begin{figure}[h]
 \centering
 \includegraphics[scale=1.1]{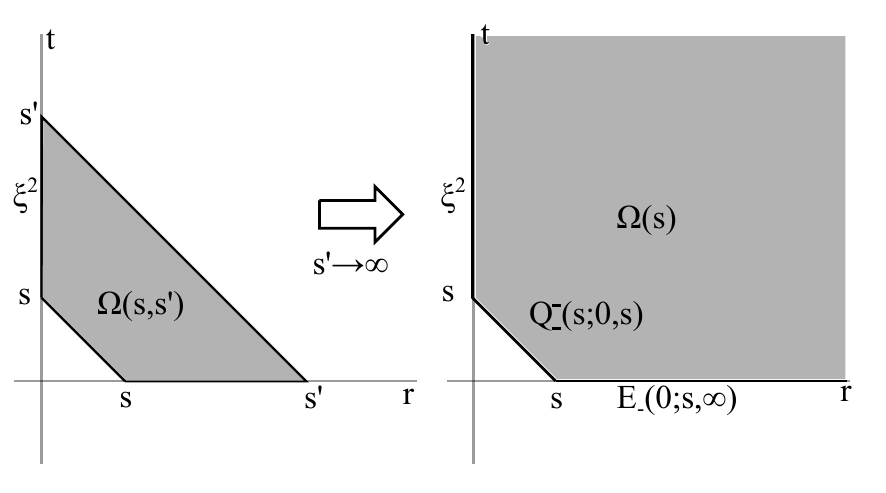}
 \caption{Illustration of regions in the proof of Proposition \ref{Weighted Morawetz Inequality}} \label{figure weightM}
\end{figure}

\begin{proof}
 We apply the energy flux formula of inward energy on the region $\Omega(s,s')=\{(r,t): s<r+t<s', t>0, r>0\}$, as shown in figure \ref{figure weightM}, and obtain ($s<s'$)
 \begin{align*}
    E_-(0,s,s')  - Q_-^-(s',0,s')-\pi \int_{s}^{s'} |\xi(t)|^2 dt+Q_-^-(s; 0,s)  - \frac{2\pi(p-1)}{p+1} \iint_{\Omega(s,s')} \frac{|w|^{p+1}}{r^p} dr dt = 0.
 \end{align*}
 Next we recall $\displaystyle \lim_{s'\rightarrow \infty} Q_-^-(s') = 0$ by Proposition \ref{a limit of Q}, make $s'\rightarrow \infty$ and obtain the energy flux formula for unbounded region $\Omega(s)=\{(r,t), r>0,t>0, r+t>s\}$.
 \[
  E_-(0,s,\infty)  -\pi \int_{s}^{\infty} |\xi(t)|^2 dt+Q_-^-(s; 0,s)  - \frac{2\pi(p-1)}{p+1} \iint_{\Omega(s)} \frac{|w|^{p+1}}{r^p} dr dt = 0.
 \]
 We move the terms with a negative sign to the other side of the identity, then plug in the definitions of $Q$ and $E$.
 \begin{align}
  &\pi \int_s^\infty |\xi(t)|^2 dt + \frac{2\pi(p-1)}{p+1} \iint_{\Omega(s)} \frac{|w(r,t)|^{p+1}}{r^p} dr dt \nonumber\\
  &\quad = \pi \int_s^\infty \left(\left|w_r(r,0)+w_t(r,0)\right|^2+\frac{2}{p+1}\cdot\frac{|w(r,0)|^{p+1}}{r^{p-1}}\right) dr + \frac{4\pi}{p+1} \int_0^s \frac{|w(s-t,t)|^{p+1}}{(s-t)^{p-1}}dt. \label{energy flux for omega s}
 \end{align}
We then multiply both sides by $a'(s)$ and integrate for $s$ from $1$ to $R\gg 1$
\begin{align*}
 & \pi \int_1^\infty (a_R(t)-1) |\xi(t)|^2 dt + \frac{2\pi(p-1)}{p+1} \iint_{\Omega_1} \frac{(a_R(r+t)-1)|w(r,t)|^{p+1}}{r^p} drdt\\
 &\qquad = \pi \int_1^\infty (a_R(r)-1) \left(\left|w_r(r,0)+w_t(r,0)\right|^2+\frac{2}{p+1}\cdot\frac{|w(r,0)|^{p+1}}{r^{p-1}}\right) dr\\
 & \qquad\qquad  + \frac{4\pi}{p+1} \iint_{1<r+t<R} \frac{a'(r+t)|w(r,t)|^{p+1}}{r^{p-1}} drdt.
\end{align*}
Here $a_R$ is a truncated version of $a$ defined by 
\[
 a_R(x) = \left\{\begin{array}{ll} a(x), & \hbox{if}\; 1\leq x<R;\\ a(R), &\hbox{if}\; x\geq R. \end{array}\right.
\]
Next we choose $s=1$ in equation \eqref{energy flux for omega s}, add it to the identity above and obtain 
\begin{align*}
 & \pi \int_1^\infty a_R(t) |\xi(t)|^2 dt + \frac{2\pi(p-1)}{p+1} \iint_{\Omega} \frac{a_R(r+t)|w(r,t)|^{p+1}}{r^p} drdt\\
 &\qquad = \pi \int_1^\infty a_R(r) \left(\left|w_r(r,0)+w_t(r,0)\right|^2+\frac{2}{p+1}\cdot\frac{|w(r,0)|^{p+1}}{r^{p-1}}\right) dr\\
 & \qquad\qquad  + \frac{4\pi}{p+1} \iint_{1<r+t<R} \frac{a'(r+t)|w(r,t)|^{p+1}}{r^{p-1}} drdt + Q_-^-(1;0,1).
\end{align*}
Now we have the key observation by the assumption on the function $a$
\begin{align*}
 \frac{4\pi}{p+1} \iint_{1<r+t<R} \frac{a'(r+t)|w(r,t)|^{p+1}}{r^{p-1}} drdt \leq & \frac{4\pi\gamma}{p+1} \iint_{1<r+t<R} \frac{a(r+t)|w(r,t)|^{p+1}}{(r+t) r^{p-1}} drdt\\
 \leq &  \frac{4\pi\gamma}{p+1} \iint_{\Omega_1} \frac{a_R(r+t)|w(r,t)|^{p+1}}{r^p} drdt.
\end{align*}
This immediately gives us
\begin{align*}
 & \pi \int_1^\infty a_R(t) |\xi(t)|^2 dt + \frac{2\pi(p-1-2\gamma)}{p+1} \iint_{\Omega} \frac{a_R(r+t)|w|^{p+1}}{r^p} drdt\\
 &\quad \leq \pi \int_1^\infty a_R(r) \left(\left|w_r(r,0)+w_t(r,0)\right|^2+\frac{2}{p+1}\cdot\frac{|w(r,0)|^{p+1}}{r^{p-1}}\right) dr + Q_-^-(1;0,1) \leq K_1.
\end{align*}
Please note that $Q_-^-(1,0,1)$ is smaller than inward energy in the ball of radius $1$, i.e. the first term in the definition of $K_1$, by the triangle law Proposition \ref{triangle law}. Finally we can make $R\rightarrow \infty$ and finish our proof.
\end{proof}

\subsection{Decay Rate of Energies}

\begin{corollary}[Weighted Morawetz]\label{weighted Morawetz}
Let $u$ be a solution as in the main theorem and $w=ru$, then we have the following decay estimates
\begin{align*}
  &\int_0^\infty t^{\kappa}|\xi(t)|^2 dt \lesssim_p K,& &\int_0^\infty \int_0^\infty \frac{(r+t)^\kappa |w(r,t)|^{p+1}}{r^p} drdt \lesssim_{p,\kappa} K,& &E_-(t)\lesssim_{p,\kappa} K t^{-\kappa}.&
 \end{align*}
Here $\xi$ is the $L^2$ function in the energy flux formula and the constant $K$ is defined in the main theorem. 
\end{corollary}
\begin{proof}
First of all, we observe the identity($r=|x|$)
\[
 w_r(r,0)+w_t(r,0) = r \partial_r u_0 + u_0 +r u_1 = |x|\left(\nabla u_0(x)\cdot \frac{x}{|x|} + \frac{u_0(x)}{|x|} + u_1(x)\right).
\]
 Thus we can apply Proposition \ref{weighted Morawetz} with $a(r) = r^\kappa$, $\gamma =\kappa$ and 
 \begin{align*}
  K_1 & = \pi \int_0^\infty \max\{1,r^\kappa\} \left(\left|w_r(r,0)+w_t(r,0)\right|^2+\frac{2}{p+1}\cdot\frac{|w(r,0)|^{p+1}}{r^{p-1}}\right) dr\\
  & = \frac{1}{4} \int_{\Rm^3} \max\{1,|x|^\kappa\} \left(\left|\nabla u_0(x)\cdot \frac{x}{|x|} + \frac{u_0(x)}{|x|} + u_1(x)\right|^2 + \frac{2}{p+1}|u_0(x)|^{p+1}\right) dx = \frac{K}{4}.
 \end{align*}
 Our conclusion includes
 \begin{align*}
  &\int_1^\infty t^{\kappa}|\xi(t)|^2 dt \lesssim_p K,& &\iint_{r>0,t>0,r+t>1} \frac{(r+t)^\kappa |w(r,t)|^{p+1}}{r^p} drdt \lesssim_{p,\kappa} K.&
 \end{align*}
 According to Proposition \ref{infinite triangle}, we also have
  \begin{align*}
  &\int_0^\infty |\xi(t)|^2 dt \lesssim_p E_-(0),& &\int_0^\infty \int_0^\infty \frac{|w(r,t)|^{p+1}}{r^p} drdt \lesssim_{p} E_-(0),&
 \end{align*}
 Combining these estimates with the fact $E_-(0) \leq K_1 = K/4$, we obtain the first two inequalities in this corollary. In order to prove the last inequality we apply Proposition \ref{infinite triangle} again
 \begin{align*}
   E_-(t)  & = \pi \int_{t}^{\infty} |\xi(t')|^2 dt' + \frac{2\pi (p-1)}{p+1} \int_{t}^{+\infty} \int_0^{+\infty} \frac{|w(r,t')|^{p+1}}{r^p} dr dt'\\
   & \lesssim_p   t^{-\kappa} \int_t^\infty (t')^{\kappa}|\xi(t')|^2 dt' + t^{-\kappa} \int_{t}^{+\infty} \int_0^{+\infty} \frac{(r+t')^\kappa |w(r,t')|^{p+1}}{r^p} dr dt' \\
   & \lesssim_{p,\kappa} Kt^{-\kappa}.
 \end{align*}
\end{proof}

\begin{remark}
 A careful review of the proof shows that the dependence of the implicit constant in the weighted Morawetz estimates on $p$ and $\kappa$ can be removed as long as $p-1-2\kappa>\delta$ for a fixed positive constant $\delta$.  This may also help us use an absolute constant in many estimates below, unless the pair $(p,\kappa)$ is closed to $(3,1)$. 
\end{remark}

\subsection{Integral of Energy in a Cylinder}

In this subsection we consider the integral of energy $E_\mp(t;0,R)$ with respect to $t$:
\[
 \pi \int_{t_0}^\infty \int_0^R \left(|w_r\pm w_t|^2 +\frac{2}{p+1}\cdot \frac{|w|^{p+1}}{r^{p-1}}\right) dr dt.
\]
This kind of integral has been considered in the classic Morawetz estimate. Roughly speaking, it tells us that the energy can stay in a given ball only for a relatively short time. This will be used to prove the $L^p L^{2p}$ bound if $\kappa>\kappa_0(p)$.  

\begin{proposition} \label{Morawetz energy cylinder}
Let $u$ be a solution to (CP1) as in the main theorem. Then
\begin{align*}
 &\int_{t_0}^\infty E_+(t;0,R) dt \lesssim_p Rt_0^{-\kappa}\varphi(t_0-R) \lesssim_{p,\kappa} KRt_0^{-\kappa},& &t_0\geq R>0;&\\
 &\int_{t_0}^\infty E_-(t;0,R) dt \lesssim_p \quad Rt_0^{-\kappa}\varphi(t_0) \quad \lesssim_{p,\kappa} KRt_0^{-\kappa},& &t_0,R>0.&
\end{align*}
Here the function $\varphi$ is defined below with $\displaystyle \lim_{t'\rightarrow +\infty} \varphi(t') = 0$.
\[
  \varphi(t') = \int_{t'}^\infty t^\kappa |\xi(t)|^2 dt + \int_{t'}^\infty \int_0^\infty \frac{(r+t)^\kappa |w(r,t)|^{p+1}}{r^p} dr dt.
\]
\end{proposition}
\begin{proof}
 Let us prove the outward case. The inward case can be solved in the same way but the calculation is easier. We start by applying the triangle law on the triangle $\Omega(t) = \{(r,t'): r>0, t'<t, t'-r>t-R\}$ with $t\geq t_0$, as shown in the left half of figure \ref{figure integralEp}.
\begin{align*}
 E_+\left(t;0,R\right) & \lesssim_p \int_{t-R}^{t} |\xi(t')|^2 dt' + \int_{t-R}^{t} \frac{|w(t'-t+R,t')|^{p+1}}{(t'-t+R)^{p-1}}dt'
  + \iint_{\Omega(t)} \frac{|w(r,t')|^{p+1}}{r^p} dr dt'.
\end{align*}

 \begin{figure}[h]
 \centering
 \includegraphics[scale=1.1]{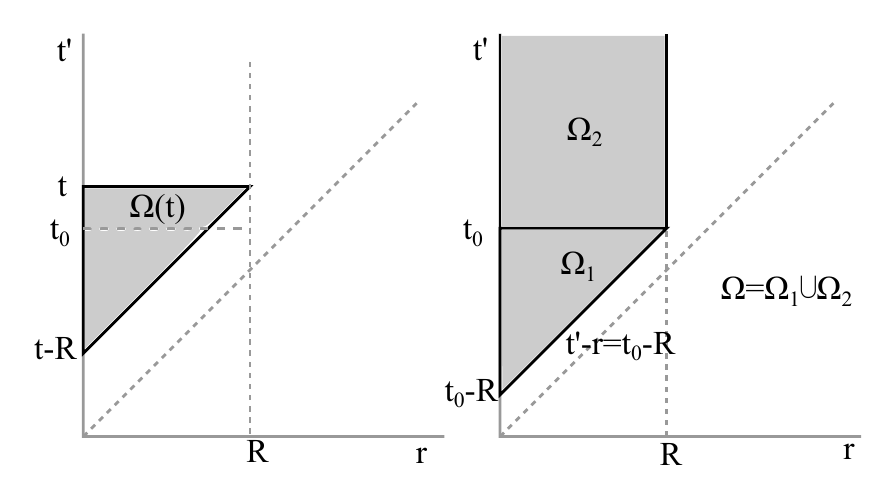}
 \caption{Illustration of proof for Proposition \ref{decay estimate outward}} \label{figure integralEp}
\end{figure}

\noindent We integrate $t$ from $t_0$ to $\infty$ and obtain
\begin{align}
 \int_{t_0}^\infty E_+(t;0,R) dt & \lesssim_p \int_{t_0-R}^{t_0} (t'-t_0+R) |\xi(t')|^2 dt' + \int_{t_0}^{\infty} R |\xi(t')|^2 dt'  + \iint_\Omega \frac{|w(r,t')|^{p+1}}{r^{p-1}} dr dt'\nonumber\\
 & \qquad + \int_{t_0}^\infty \left(\iint_{\Omega(t)} \frac{|w(r,t')|^{p+1}}{r^p} dr dt' \right)dt.\label{integral of E inequality}
\end{align}
The region $\Omega$ is defined as $\Omega = \{(r,t'): 0<r<R, t'-r>t_0-R\}$. We calculate the last term above in details
\begin{align*}
 \int_{t_0}^\infty \left(\iint_{\Omega(t)} \frac{|w(r,t')|^{p+1}}{r^p} dr dt' \right)dt & = \int_{t_0}^\infty \left(\iint_{\Omega} \chi_{\Omega(t)}(r,t') \frac{|w(r,t')|^{p+1}}{r^p} dr dt' \right)dt \\
 & = \iint_{\Omega} \frac{|w(r,t')|^{p+1}}{r^p}\left(\int_{t_0}^\infty \chi_{\Omega(t)}(r,t') dt\right) dr dt'
\end{align*}
The function $\chi_{\Omega(t)}(r,t')$ is the characteristic function of $\Omega(t)$. The value of $\chi_{\Omega(t)}(r,t')=1$ if and only if $t'<t$ and $t'-r>t-R$ both hold. This is equivalent to saying $t\in (t',t'-r+R)$. Thus we have
\[ 
 \int_{t_0}^\infty \chi_{\Omega(t)}(r,t') dt = \left\{\begin{array}{ll} t'-r+R-t_0, & \hbox{if}\; t'<t_0;\\ R-r, & \hbox{if}\; t'\geq t_0.\end{array}\right.
\]
This enable us to write the last term in the right hand of \eqref{integral of E inequality} in details. In summary we have
\begin{align*} 
 \int_{t_0}^\infty E_+(t;0,R) dt & \lesssim_p \int_{t_0-R}^{t_0} (t'-t_0+R) |\xi(t')|^2 dt' + \int_{t_0}^{\infty} R |\xi(t')|^2 dt'  + \iint_\Omega \frac{|w(r,t')|^{p+1}}{r^{p-1}} dr dt'\\
 & +\iint_{\Omega_1} \frac{(t'-r+R-t_0)|w(r,t')|^{p+1}}{r^p} dr dt' + \iint_{\Omega_2} \frac{(R-r)|w(r,t')|^{p+1}}{r^p} dr dt'
\end{align*}
Here we divide the region $\Omega$ into two parts: $\Omega_1 = \{(r,t'): 0<r<R,  t'-r>t_0-R, t'<t_0\}$ and $\Omega_2 = \{(r,t'): 0<r<R, t' \geq t_0\}$, as illustrated in the right half of figure \ref{figure integralEp}. Next we split the integral of $|w|^{p+1}/r^{p-1}$ over $\Omega$ into integrals over $\Omega_1$ and $\Omega_2$, and then combine integrals over the same region together
\begin{align}
 & \int_{t_0}^\infty E_+(t;0,R) dt  \lesssim_p \int_{t_0-R}^{t_0} \frac{t'-t_0+R}{(t')^\kappa} (t')^\kappa |\xi(t')|^2 dt' + \int_{t_0}^{\infty} R |\xi(t')|^2 dt' \nonumber\\
 &\qquad \quad +\iint_{\Omega_1} \frac{t'-t_0+R}{(t'+r)^\kappa}\cdot \frac{(t'+r)^\kappa |w(r,t')|^{p+1}}{r^p} dr dt' + \iint_{\Omega_2} \frac{R|w(r,t')|^{p+1}}{r^p} dr dt'. \label{integral inequality cylinder prepared}
\end{align}
Now we observe the fact that $(t'-t_0+R)/(t')^\kappa$ is an increasing function of $t'>0$. Thus
\[
 \frac{t'-t_0+R}{(t'+r)^\kappa} \leq \frac{t'-t_0+R}{(t')^\kappa} \leq \frac{t_0-t_0+R}{t_0^{\kappa}} = Rt_0^{-\kappa}.
\]
We plug this upper bound in our estimate \eqref{integral inequality cylinder prepared} and obtain
\[
 \int_{t_0}^\infty E_+(t;0,R) dt \lesssim_p Rt_0^{-\kappa} \varphi(t_0-R). 
\] 
Finally we recall weighted Morawetz estimate $\varphi(t_0-R) \lesssim_{p,\kappa} K$ and finish the proof.
\end{proof}

\subsection{A decay estimate of outward energy}

We can also prove a decay estimate on outward energy, but only for local outward energy $E_+(t;0,r)$ with $r<t$. We will not consider the case $r>t$, because in this case $E_+(t,0,r)$ includes a major part of energy, which concentrates around the light cone $r=t$.
\begin{proposition} \label{decay estimate outward}
 Let $u$ be a solution as in the main theorem, we have the following estimate on the outward energy ($0<r<t$)
 \[
  E_+(t;0,r) \lesssim_p \frac{t^{1-\kappa}\varphi\left(\frac{t-r}{2}\right)}{t-r} \lesssim_{p,\kappa} \frac{Kt^{1-\kappa}}{t-r}.
 \]
 Here $\varphi$ is defined in Proposition \ref{Morawetz energy cylinder}. 
\end{proposition}

\begin{proof} Given any $t'>t$, we may apply the energy flux formula on the region $\Omega(t')=\{(\bar{r},\bar{t}): t<\bar{t}<t', 0<\bar{r}<r+\bar{t}-t\}$, as illustrated in figure \ref{figure upperEp},  and write
\[
 E_+(t';0,t'-t+r) = E_+(t;0,r) + \hbox{three nonnegative terms}.
\]
Therefore we always have 
\begin{equation}
 E_+(t;0,r) \leq E_+(t';0,t'-t+r) \leq E_+\left(t';0,\frac{t+r}{2}\right), \quad t<t'<t+\frac{t-r}{2} \label{upper bd}
\end{equation}
We may integrate \eqref{upper bd} for $t'\in (t,t+\frac{t-r}{2})$, apply Proposition \ref{Morawetz energy cylinder} and obtain
\begin{align*}
 \frac{t-r}{2}\cdot E_+(t;0,r) & \leq \int_{t}^{t+\frac{t-r}{2}} \!\!E_+\left(t';0,\frac{t+r}{2}\right)dt' \leq \int_t^{\infty} E_+\left(t',0,\frac{t+r}{2}\right) dt' \\
 & \lesssim_p \frac{t+r}{2}\cdot t^{-\kappa}\varphi\left(\frac{t-r}{2}\right)\lesssim_p t^{1-\kappa} \varphi\left(\frac{t-r}{2}\right)\lesssim_{p,\kappa} K t^{1-\kappa}. 
\end{align*}
This finishes the proof.

 \begin{figure}[h]
 \centering
 \includegraphics[scale=1.1]{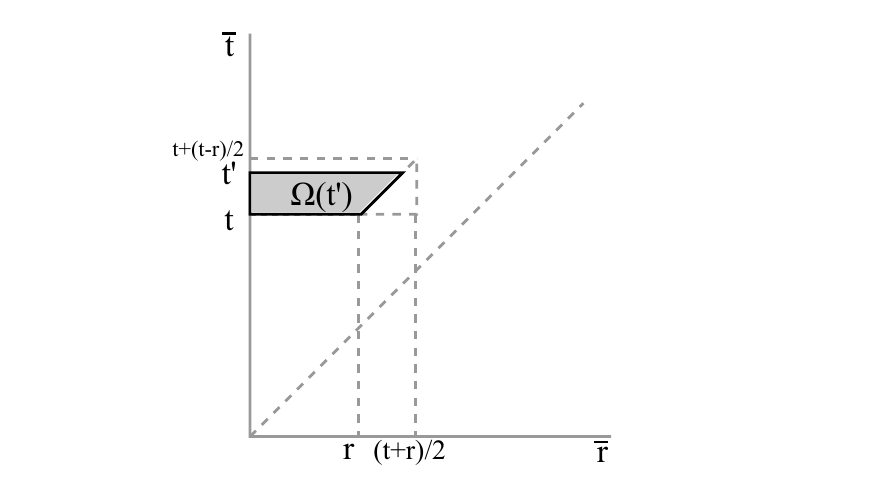}
 \caption{Illustration of proof for Proposition \ref{decay estimate outward}} \label{figure upperEp}
\end{figure}
\end{proof}

\section{Proof of Main Theorem}

In this section we prove our main theorem in details.

\subsection{Proof of Scattering Part}
If the solution did not scatter, we would have the following limit for any fixed $\beta < \frac{2(p-2)}{p+1} = 1 - \kappa_0(p)$, according to part (c) of Theorem \ref{energy distribution}.
\begin{equation} \label{location of retarded energy}
 \lim_{t\rightarrow \infty} E(t;t/2,t-t^\beta) >0.
\end{equation}
We need to show a contradiction. Let us start by
\begin{proposition} \label{decay of outward energy}
 Let $u$ be a solution as in the main theorem. Given any $c>0$, we have
 \[
  \lim_{t \rightarrow +\infty} E(t;0,t-ct^{1-\kappa}) = 0.
 \]
\end{proposition}
\begin{proof}
 Since we know the inward energy converges to zero as $t\rightarrow +\infty$. This is sufficient to show
 \[
  \lim_{t \rightarrow +\infty} E_+(t;0,t-ct^{1-\kappa}) = 0.
 \]
 It immediately follows Proposition \ref{decay estimate outward}
 \[
  E_+(t;0,t-ct^{1-\kappa}) \lesssim_p \frac{t^{1-\kappa}\varphi\left(ct^{1-\kappa}/2\right)}{ct^{1-\kappa}} = \frac{1}{c} \varphi(ct^{1-\kappa}/2) \rightarrow 0.
 \]
\end{proof}
\paragraph{Proof of scattering} Now we are ready to prove part (a) of the main theorem. If $\kappa > \kappa_0(p)$, we can choose $\beta \in (1-\kappa,\frac{2(p-2)}{p+1})$. Thus we have $0<t/2<t-t^\beta <t-ct^{1-\kappa}$ for sufficiently large $t$. As a result, the conclusion of Proposition \ref{decay of outward energy} contradicts with \eqref{location of retarded energy}.
On the other hand, if $\kappa=\kappa_0(p)$, we have 
\begin{align*}
 E(t;t/2,t-t^\beta) & \leq E(t;0,t-ct^{1-\kappa_0}) + E_+(t;t-ct^{1-\kappa_0}, t-t^\beta) + E_-(t;t-ct^{1-\kappa_0}, t-t^\beta)\\
 & \leq E(t;0,t-ct^{1-\kappa_0}) + \pi \int_{t-ct^{1-\kappa_0}}^{t-t^\beta}|w_r-w_t|^2 dr + 2E_-(t;t-ct^{1-\kappa_0}, t-t^\beta)
\end{align*}
Here $c$ is an arbitrary small constant. The first and last term above converge to zero as $t\rightarrow +\infty$ by Proposition \ref{decay of outward energy} and the decay estimate $E_-(t) \lesssim_{p,\kappa} K  t^{-\kappa}$, respectively. Proposition \ref{limits of wr pm wt} gives an upper bound of the second term above
\begin{align*}
  \pi \int_{t-ct^{1-\kappa_0}}^{t-t^\beta}|w_r-w_t|^2 dr & = \pi \int_{t^\beta}^{ct^{1-\kappa_0}} \left|(w_r-w_t)(t-\tau, t)\right|^2 d\tau\\
  & \lesssim_1 \int_{t^\beta}^{ct^{1-\kappa_0}} \left|(w_r-w_t)(t-\tau, t)-g_+(\tau)\right|^2 d\tau + \int_{t^\beta}^{ct^{1-\kappa_0}}|g_+(\tau)|^2 d\tau\\
  & \lesssim_{p,E} \int_{t^\beta}^{ct^{1-\kappa_0}} \left[(t-\tau)^{-\frac{p-2}{p+1}}\right]^2 d\tau + \int_{t^\beta}^{ct^{1-\kappa_0}} |g_+(\tau)|^2 d\tau\\
  & \lesssim_{p,E} c + \int_{t^\beta}^{ct^{1-\kappa_0}} |g_+(\tau)|^2 d\tau.
\end{align*}
In summary, we have for any $c>0$
\[
 \limsup_{t\rightarrow \infty} E(t;t/2,t-t^\beta) \lesssim_{p,E} c \quad \Rightarrow \quad \limsup_{t\rightarrow \infty} E(t;t/2,t-t^\beta) = 0.
\] 
This contradicts with \eqref{location of retarded energy} and finishes the proof.

\subsection{Proof of Finite $L^p L^{2p}$ Norm}
Throughout this subsection, we assume $u$ is a solution as in the main theorem. The decay rate satisfies $\kappa>\kappa_0(p)$. First of all, let us recall the local theory
\begin{lemma} \label{local theory}
 Let $(u_0,u_1) \in \dot{H}^1 \times L^2(\Rm^3)$ be initial data. Then the Cauchy problem (CP1) has a unique solution $u$ in the time interval $[0,T]$ with 
 \begin{align*}
   &(u,u_t) \in C([0,T];\dot{H}^1 \times L^2(\Rm^3));& &u \in L^{\frac{2p}{p-3}}L^{2p} ([0,T]\times \Rm^3).&
 \end{align*}
 Here the minimal time length of existence $T = C_p \|(u_0,u_1)\|_{\dot{H}^1\times L^2}^{\frac{-2(p-1)}{5-p}}$. In addition, we have 
 \[
  \|u\|_{L^{\frac{2p}{p-3}}L^{2p} ([0,T]\times \Rm^3)} \lesssim_p \|(u_0,u_1)\|_{\dot{H}^1\times L^2}.
 \]
\end{lemma}
\begin{proof}
 Let us consider the operator $\mathbf{T}$ defined for $ u \in L^{\frac{2p}{p-3}} L^{2p} ([0,T]\times \Rm^3)$
 \[
  \mathbf{T} u = \mathbf{S}(t) (u_0,u_1) - \int_0^t \frac{\sin(t-s)\sqrt{-\Delta}}{\sqrt{-\Delta}} |u(\cdot, s)|^{p-1} u(\cdot, s) ds
 \]
 and apply a standard fixed-point argument. By the classic Strichartz estimate given in \cite{strichartz}, we have 
 \begin{align*}
  \|\mathbf{T} u\|_{L^{\frac{2p}{p-3}}L^{2p} ([0,T]\times \Rm^3)} & \leq C\|(u_0,u_1)\|_{\dot{H}^1\times L^2} + C \||u|^{p-1}u\|_{L^1 L^2([0,T]\times \Rm^3)}\\
  & \leq C\|(u_0,u_1)\|_{\dot{H}^1\times L^2} + C T^{\frac{5-p}{2}}\|u\|_{L^{\frac{2p}{p-3}}L^{2p} ([0,T]\times \Rm^3)}^p;
 \end{align*}
 \begin{align*}
  \|\mathbf{T} u - \mathbf{T} v\|_{L^{\frac{2p}{p-3}}L^{2p} ([0,T]\times \Rm^3)} & \leq C\||v|^{p-1}v- |u|^{p-1}u\|_{L^1 L^2([0,T]\times \Rm^3)} \\
  & \leq CT^{\frac{5-p}{2}} \|u-v\|_{L^{\frac{2p}{p-3}}L^{2p}} \left(\|u\|_{L^{\frac{2p}{p-3}}L^{2p}}^{p-1} + \|v\|_{L^{\frac{2p}{p-3}}L^{2p}}^{p-1}  \right).
 \end{align*}
 When $T \leq C_p \|(u_0,u_1)\|_{\dot{H}^1\times L^2}^{\frac{-2(p-1)}{5-p}}$ with a small constant $C_p$, the operator $\mathbf{T}$ becomes a contraction map from the metric space $\{u: \|u\|_{L^{\frac{2p}{p-3}}L^{2p} ([0,T]\times \Rm^3)}\leq  2C\|(u_0,u_1)\|_{\dot{H}^1\times L^2} \} $ to itself. A standard fixed-point argument immediately prove this lemma. 
\end{proof}
\begin{remark}
 This local theory is much different from the local theory with initial data in the critical space. In the critical case, unless the initial data is small, the minimal time of existence given by local theory depends on not only the size, but also the profile of initial data. Please see, for instance, \cite{ls} for more details.
\end{remark}
\paragraph{Small time part} Since we know
\[
 \|(u(\cdot,t),u_t(\cdot,t))\|_{\dot{H}^1\times L^2} \lesssim_1 E^{1/2}
\]
is uniformly bounded, we may split the time interval $[0,2]$ into finite small intervals, so that in each small interval Lemma \ref{local theory} applies. This gives
\[
 \|u\|_{L^p L^{2p}([0,2]\times \Rm^3)} \lesssim_p \|u\|_{L^\frac{2p}{p-3} L^{2p}([0,2]\times \Rm^3)} < +\infty.
\]
\paragraph{Large time part} Now we need to deal with the case $t>2$. We write the solution $u$ as a sum of three parts 
\[
 u(x,t)  = u(x,t) \chi_{|x|>t/4}(x,t) + u(x,t) \chi_{1/4 \leq |x| \leq t/4}(x,t) + u(x,t) \chi_{|x|<1/4}(x,t).
\]
Here $\chi$'s are characteristic functions of the indicated regions. We call these three parts large radius, medium radius and small radius parts, respectively. 

\paragraph{Large radius part} Let $t > 2$ and $|x|>t/4$. This is the major part, as we mentioned in the introduction. By Lemma \ref{local upper bound pointwise} and the decay estimate of inward energy $E_-(t) \lesssim_{p,\kappa} Kt^{-\kappa}$, we have
\[
 |w(|x|,t)| \lesssim_{p,\kappa} |x|^\frac{p-1}{p+3} E^{\frac{1}{p+3}} (Kt^{-\kappa})^\frac{1}{p+3}\Rightarrow |u(x,t)| \lesssim_{p,\kappa,E,K} |x|^{-\frac{4}{p+3}} t^{-\frac{\kappa}{p+3}} 
\]
Therefore we have 
\begin{align*}
 \int_{|x|>t/4} |u(x,t)|^{2p} dx &\lesssim_{p,\kappa,E,K} \int_{|x|>t/4} |u(x,t)|^{p+1} \left(|x|^{-\frac{4}{p+3}} t^{-\frac{\kappa}{p+3}}\right)^{p-1} dx\\
 &\lesssim_{p,\kappa,E,K} \int_{|x|>t/4} t^\kappa |u(x,t)|^{p+1} \cdot |x|^{-\frac{4(p-1)}{p+3}} t^{-\frac{\kappa(2p+2)}{p+3}} dx\\
 &\lesssim_{p,\kappa,E,K} t^{-\frac{4(p-1)+\kappa(2p+2)}{p+3}}.
 \end{align*}
 Here we use the estimate $\int_{\Rm^3} t^\kappa |u(x,t)|^{p+1} dx \lesssim_1 t^\kappa E_-(t)\lesssim_{p,\kappa} K$. As a result, we have ($t_0>1$) 
 \begin{align}
  \|u \chi_{|x|>t/4}\|_{L^p L^{2p} ([t_0,\infty)\times \Rm^3)}^p &= \int_{t_0}^\infty \left(\int_{|x|>t/4} |u(x,t)|^{2p} dx\right)^{1/2} dt\nonumber \\
  &\lesssim_{p,\kappa,E,K} \int_{t_0}^\infty  t^{-\frac{2(p-1)+\kappa(p+1)}{p+3}} dt\nonumber \\
  &\lesssim_{p,\kappa,E,K} t_0^{-\frac{\kappa(p+1)+(p-5)}{p+3}} = t_0^{-\frac{p+1}{p+3}(\kappa-\kappa_0(p))}. \label{large radius part}
 \end{align}
 Therefore we have 
 \[
   \|u \chi_{|x|>t/4}\|_{L^p L^{2p} ([2,\infty)\times \Rm^3)} < + \infty.
 \]
 
\paragraph{Medium radius part} We need to use the following lemma

\begin{lemma} \label{piece estimate}
 Assume that $0\leq \alpha<\beta\leq 1$ are constants. Let $t_0 \geq 1$ be a time. Then
 \[
  \|u\|_{L^p L^{2p} ([t_0,2t_0]\times \{x\in \Rm^3: t_0^\alpha/4 \leq |x| \leq t_0^\beta/2\})}^p \lesssim_{p,\kappa} K^{\frac{2p}{p+3}} t_0^{\beta-\frac{(5p-9)}{2(p+3)}\alpha-\frac{2p}{p+3}\kappa}.
 \]
\end{lemma}
\begin{proof}
 We recall the pointwise estimates on $w$ given in Lemma \ref{local upper bound pointwise} for $r\in (0,t_0^\beta/2)$ and $t \in (t_0,2t_0)$
\[
 |w(r,t)|  \lesssim_p r^{\frac{p-1}{p+3}}\left(\int_0^r |w_r(s,t)|^2 ds\right)^{\frac{1}{p+3}} \left(\int_0^r \frac{|w(s,t)|^{p+1}}{s^{p-1}} ds\right)^{\frac{1}{p+3}}
 \lesssim_p r^{\frac{p-1}{p+3}} E(t;0,t_0^\beta/2)^{\frac{2}{p+3}}.
\]
Since $w=ru$, we have the following estimate
\[
  |u(x,t)|  \lesssim_p |x|^{-\frac{4}{p+3}} E(t;0,t_0^\beta/2)^{\frac{2}{p+3}}.
\]
The combination of Corollary \ref{weighted Morawetz} and Proposition \ref{decay estimate outward} gives an estimate $E(t;0,t_0^\beta/2) \lesssim_{p,\kappa} K t_0^{-\kappa}$. As a result, we have 
\[
 |u(x,t)|^{2p} \lesssim_{p} |x|^{-\frac{8p}{p+3}} E(t;0,t_0^\beta/2)^{\frac{4p}{p+3}}
 \lesssim_{p,\kappa}  |x|^{-\frac{8p}{p+3}} K^{\frac{2(p-3)}{p+3}} t_0^{-\frac{2(p-3)}{p+3}\kappa} E(t;0,t_0^\beta/2)^2.
\]
We integrate in $x$,
\begin{align*}
 \int_{t_0^\alpha/4\leq |x| \leq t_0^\beta/2} |u(x,t)|^{2p} dx \lesssim_{p,\kappa} K^{\frac{2(p-3)}{p+3}} t_0^{-\frac{5p-9}{p+3}\alpha-\frac{2(p-3)}{p+3}\kappa} E(t;0,t_0^\beta/2)^2
\end{align*}
Therefore we have 
\begin{align*}
 \|u\|_{L^p L^{2p} ([t_0,2t_0]\times \{x\in \Rm^3: t_0^\alpha/4 \leq |x| \leq t_0^\beta/2\})}^p & = \int_{t_0}^{2t_0}\left(\int_{t_0^\alpha/4\leq |x| \leq t_0^\beta/2} |u(x,t)|^{2p} dx\right)^{1/2} dt \\
 & \lesssim_{p,\kappa} K^{\frac{p-3}{p+3}} t_0^{-\frac{(5p-9)}{2(p+3)}\alpha-\frac{p-3}{p+3}\kappa} \int_{t_0}^{2t_0}  E(t;0,t_0^\beta/2) dt\\
 & \lesssim_{p,\kappa} K^{\frac{2p}{p+3}} t_0^{\beta-\frac{(5p-9)}{2(p+3)}\alpha-\frac{2p}{p+3}\kappa} 
\end{align*}
In the final step, we use Proposition \ref{Morawetz energy cylinder}.
\end{proof}
\noindent we can rewrite the exponent of $t_0$ above by
\begin{align*}
 & \beta-\frac{(5p-9)}{2(p+3)}\alpha-\frac{2p}{p+3}\kappa \\
 & \qquad= \frac{3(5-p)}{2(p+3)}(\beta-1) + \frac{5p-9}{2(p+3)}(\beta-\alpha) - \frac{2p}{p+3}(\kappa-\kappa_0(p)) -\frac{(p-3)(5-p)}{2(p+1)(p+3)}.
\end{align*}
We can always choose a finite sequence of numbers $0=\alpha_0 < \alpha_1 < \cdots < \alpha_n = 1$ so that
\[
 -\gamma(j) \doteq \alpha_j-\frac{(5p-9)}{2(p+3)}\alpha_{j-1}-\frac{2p}{p+3}\kappa < -\frac{p+1}{p+3}(\kappa-\kappa_0(p)), \qquad j=1,2,\cdots, n. 
\]
These constants depend on $p,\kappa$ but nothing else. An application of Lemma \ref{piece estimate} gives
\begin{align*}
 & \left\|u\chi_{t^{\alpha_{j-1}}/4<|x|<t^{\alpha_j}/4}\right\|_{L^p L^{2p}([2^i t_0,2^{i+1}t_0]\times \Rm^3)}^p \\
 & \qquad \leq \|u\|_{L^p L^{2p} ([2^i t_0,2^{i+1}t_0]\times \{x\in \Rm^3: (2^i t_0)^{\alpha_{j-1}}/4 \leq |x| \leq (2^i t_0)^{\alpha_j}/2\})}^p \\
 &\qquad \lesssim_{p,\kappa} K^{\frac{2p}{p+3}} (2^i t_0)^{-\gamma(j)},
\end{align*}
for any $t_0\geq 1$, $j=1,2,\cdots,n$ and $i=0,1,2,\cdots$. Since $-\gamma(j)<-\frac{p+1}{p+3}(\kappa-\kappa_0(p))$, we can take a sum 
\[
 \left\|u\chi_{t^{\alpha_{j-1}}/4<|x|<t^{\alpha_j}/4}\right\|_{L^p L^{2p}([t_0,\infty)\times \Rm^3)}^p \lesssim_{p,\kappa} \sum_{i=0}^\infty  K^{\frac{2p}{p+3}} (2^i t_0)^{-\gamma(j)} \lesssim_{p,\kappa} K^{\frac{2p}{p+3}} t_0^{-\gamma(j)}.
\]
Finally we take a sum for $j=1,2,\cdots, n$ and obtain
\begin{align*}
 \left\|u\chi_{1/4<|x|<t/4}\right\|_{L^p L^{2p}([t_0,\infty)\times \Rm^3)}^p & \lesssim_{p,n} \sum_{j=1}^n \left\|u\chi_{t^{\alpha_{j-1}}/4<|x|<t^{\alpha_j}/4}\right\|_{L^p L^{2p}([t_0,\infty)\times\Rm^3)}^p\\
 & \lesssim_{p,\kappa} K^{\frac{2p}{p+3}} \sum_{j=1}^n t_0^{-\gamma(j)}.
\end{align*}
Let us recall $-\gamma(j) <-\frac{p+1}{p+3}(\kappa-\kappa_0(p))$ and the fact that the number $n$ depends on $p,\kappa$ only. Thus we have 
\begin{equation}
  \left\|u\chi_{1/4<|x|<t/4}\right\|_{L^p L^{2p}([t_0,\infty)\times\Rm^3)}^p \lesssim_{p,\kappa} K^{\frac{2p}{p+3}} t_0^{-\frac{p+1}{p+3}(\kappa-\kappa_0(p))}. \label{medium radius part}
\end{equation}
This deals with the medium radius part.

\begin{remark}
 A brief review of the proof shows that the medium radius part decays no slower than $t_0^{-\eta}$ as $t_0\rightarrow +\infty$ as long as 
 \[
  \eta < \eta_0(p,\kappa) = \frac{2p}{p+3}(\kappa-\kappa_0(p)) + \frac{(p-3)(5-p)}{2(p+1)(p+3)}.
 \]
 If $3<p<5$, this gives us a power-type decay of the medium radius part even if $\kappa$ is slightly smaller than $\kappa_0(p)$, as long as $\eta_0(p,\kappa)>0$.
\end{remark}

\paragraph{Small radius part} We start by considering a local Strichartz-type estimate specialized in dealing with small radius part of $u$.
\begin{lemma} \label{app local}
 Let $u$ be as in the main theorem. Then there exists a constant $T \in (0,1/4]$, which is determined by $p$,$\kappa$ and $K$ but nothing else, so that for any $t_0\geq 2$, we have
 \[
  \|u\|_{L^p L^{2p} ([t_0-T,t_0+T]\times B(0,1/4))} \lesssim_p E(t_0; 0,1)^{1/2}.
 \]
\end{lemma}
\begin{proof}
 By Corollary \ref{weighted Morawetz} and Proposition \ref{decay estimate outward}, we have a uniform upper bound on the local energy
 \[
  E(t_0; 0,1) \lesssim_{p,\kappa} K t_0^{-\kappa} \lesssim_{p,\kappa} K, \quad t_0\geq 2.
 \]
 Let us consider the corresponding solution $v$ to (CP1) with the following initial data at time $t_0$
 \begin{align*}
  &v(x,t_0) = v_0(x) = u(x,t_0)\phi(x),& &v_t(x,t_0) = v_1(x) = u_t(x,t_0) \phi(x).& 
 \end{align*}
 Here $\phi(x)$ is a smooth, radial cut-off function satisfying
 \begin{align*}
  &0\leq \phi(x) \leq 1;& &\phi(x)=1, |x|<1/2;& &\phi(x) =0, |x|\geq 1.& 
 \end{align*}
 A straightforward calculation shows
 \begin{align*}
   \int_{\Rm^3} \left(|\nabla v_0|^2 + |v_1|^2\right) dx & \lesssim_1 \int_{B(0,1)} \left(|\nabla u(x,t_0)|^2 + |u_t(x,t_0)|^2 + |\nabla \phi(x)|^2 |u(x,t_0)|^2 \right) dx\\
   & \lesssim_1 E(t_0;0,1) \lesssim_{p,\kappa} K.
 \end{align*}
According to Lemma \ref{local theory}, there exists a constant $T=T(p,\kappa,K)$, so that 
\[
 \|v\|_{L^{\frac{2p}{p-3}} L^{2p} ([t_0-T,t_0+T]\times \Rm^3)} \lesssim_p \|(v_0,v_1)\|_{\dot{H}^1 \times L^2} \lesssim_p E(t_0;0,1)^{1/2}
\]
Here the time interval has been extended in both two time directions, because the wave equation is time-reversible. Without loss of generality, we can also assume $T\leq 1/4$. Now we observe 
\[
 v(x,t_0) = u(x,t_0),\;\; v_t(x,t_0) = u_t(x,t_0),\;\; |x|<1/2.
\] 
Since the wave equation has a finite speed of propagation, we immediately obtain ($T\leq 1/4$)
\[
 \|u\|_{L^{\frac{2p}{p-3}} L^{2p} ([t_0-T,t_0+T]\times B(0,1/4))} = \|v\|_{L^{\frac{2p}{p-3}} L^{2p} ([t_0-T,t_0+T]\times B(0,1/4))} \lesssim_p E(t_0;0,1)^{1/2}.
\]
Finally we are able to finish the proof by the embedding $L^{\frac{2p}{p-3}} \hookrightarrow L^{p}$.
\end{proof}
\noindent Now we are ready to give an upper bound on the small radius part. Let $T$ be the constant given above. According to Proposition \ref{Morawetz energy cylinder}, for $t_0\geq 2$ we have 
\[
 \int_{t_0}^\infty E(t;0,1) dt \lesssim_{p,\kappa} K t_0^{-\kappa}. 
\]
Next we observe
\[ 
 \int_{t_0}^{t_0+T} \left(\sum_{n=0}^\infty E(t+nT; 0,1)\right) dt = \int_{t_0}^\infty E(t;0,1) dt.
\]
It immediately follows that there exist $t' \in [t_0,t_0+T)$, so that 
\[
 \sum_{n=0}^\infty E(t'+nT; 0,1) \leq \frac{1}{T} \int_{t_0}^\infty E(t;0,1) dt \lesssim_{p,\kappa,K} t_0^{-\kappa}.
\]
Now we can apply Lemma \ref{app local}, keep the fact $E(t;0,1)\lesssim_{p,\kappa,K} 1$ in mind  and calculate
\begin{align*}
 \|u\|_{L^p L^{2p}([t_0,\infty)\times B(0,1/4))}^p &= \sum_{n=0}^\infty \|u\|_{L^p L^{2p}([t_0+nT,t_0+nT+T)\times B(0,1/4))}^p\\
 & \leq \sum_{n=0}^\infty \|u\|_{L^p L^{2p}([t'+nT-T,t'+nT+T]\times B(0,1/4))}^p\\
 & \lesssim_p \sum_{n=0}^\infty E(t'+nT;0,1)^{p/2}\\
 & \lesssim_{p,\kappa,K} \sum_{n=0}^\infty E(t'+nT;0,1)\\
 & \lesssim_{p,\kappa,K} t_0^{-\kappa}.
\end{align*}
This is equivalent to 
\begin{equation}
 \|u\chi_{|x|<1/4}\|_{L^p L^{2p}([t_0,\infty)\times \Rm^3)}^p \lesssim_{p,\kappa,K} t_0^{-\kappa}, \label{small radius part}
\end{equation}
thus finishes the proof of small radius part.  

\begin{remark}
 The norm of small radius part always decays like $t_0^{-\kappa}$. In summary, when $\kappa<\kappa_0(p)$, the major obstacle to proof of scattering results is insufficient decay rate of $u$ around the light cone $|x|=t$. 
\end{remark}

\paragraph{Convergence rate of the scattering} Let us collect the upper bounds in \eqref{large radius part}, \eqref{medium radius part}, \eqref{small radius part} and conclude
\[
 \left\|-|u|^{p-1}u\right\|_{L^1 L^2([t_0,\infty)\times \Rm^3)} = \|u\|_{L^p L^{2p} ([t_0,\infty)\times \Rm^3)}^p \lesssim t_0^{-\frac{p+1}{p+3}(\kappa-\kappa_0(p))}.
\]
Next we observe ($t'>t$, $(v_0,v_1)$ is defined in part (a) of main theorem)
\begin{align*}
 & \left\|\begin{pmatrix} u(\cdot,t)\\ u_t(\cdot,t)\end{pmatrix} - \mathbf{S}_L(t) \begin{pmatrix} v_0\\ v_1\end{pmatrix} \right\|_{\dot{H}^1\times L^2} \\
  = &\left\|\mathbf{S}_L(t'-t) \begin{pmatrix} u(\cdot,t)\\ u_t(\cdot,t)\end{pmatrix} - \mathbf{S}_L(t') \begin{pmatrix} v_0\\ v_1\end{pmatrix} \right\|_{\dot{H}^1\times L^2}\\
  \leq &\left\|\begin{pmatrix} u(\cdot,t')\\ u_t(\cdot,t')\end{pmatrix} - \mathbf{S}_L(t') \begin{pmatrix} v_0\\ v_1\end{pmatrix} \right\|_{\dot{H}^1\times L^2}
 +\left\|\begin{pmatrix} u(\cdot,t')\\ u_t(\cdot,t')\end{pmatrix} - \mathbf{S}_L(t'-t) \begin{pmatrix} u(\cdot,t)\\ u_t(\cdot,t)\end{pmatrix} \right\|_{\dot{H}^1\times L^2}\\
  \leq &\left\|\begin{pmatrix} u(\cdot,t')\\ u_t(\cdot,t')\end{pmatrix} - \mathbf{S}_L(t') \begin{pmatrix} v_0\\ v_1\end{pmatrix} \right\|_{\dot{H}^1\times L^2} 
 +C \left\|-|u|^{p-1}u\right\|_{L^1 L^2([t,t']\times \Rm^3)}.
\end{align*}
Here we apply the classic Strichartz estimate. Finally we make $t'\rightarrow +\infty$, discard the first term in the right hand side of inequality by scattering and finally obtain
\[
 \left\|\begin{pmatrix} u(\cdot,t)\\ u_t(\cdot,t)\end{pmatrix} - \mathbf{S}_L(t) \begin{pmatrix} v_0\\ v_1\end{pmatrix} \right\|_{\dot{H}^1\times L^2} \lesssim t^{-\frac{p+1}{p+3}(\kappa-\kappa_0(p))}.
\]

\section{Appendix} 
In this final appendix we give an explicit example of scattering solutions which is not covered by the scattering theory in the critical Sobolev space. We start by 
\begin{lemma} \label{estimate of nonlinear Cbeta}
Let $\beta = 1 - \frac{2}{p-1}$. Given any $(r',t') \in \Rm^+ \times \Rm^+$ with $r'>t'$, we define $\Omega(r',t') = \{(r,t): r+t<t'+r', t-r<t'-r',t>0\}$ to be a triangle region. If the inequality $|w(r,t)| \leq A r^\beta$ holds for all points $(r,t)$ in $\Omega(r',t')$, then we have
\[
   I = \int_{\Omega(r',t')} \frac{|w(r,t)|^p}{r^{p-1}} drdt \lesssim_p A^p (r')^\beta.
\]
\end{lemma}
\begin{proof}
 First of all, we plug in the upper bound of $|w(r,t)|$ and obtain
 \[
  I  \leq \int_{\Omega(r',t')} A^p r^{\beta-2} drdt.
 \]
 We enlarge the region of integral to $\Omega'(r',t') = \{(r,t): t'-r'<r+t<r'+t', t-r<t'-r'\}$, as shown in figure \ref{figure triangle integral} and apply the change of variables $r=\frac{s-\tau}{2}$, $t = \frac{s+\tau}{2}$. The region $\Omega'(r',t')$ corresponds to the region $\{(s,\tau): t'-r'<s<r'+t', \tau<t'-r'\}$ in the $s$-$\tau$ space. The conclusion then follows a straight-forward calculation
\begin{align*}
 I & \leq \int_{\Omega'(r',t')} A^p r^{\beta-2} drdt = \frac{1}{2} \int_{t'-r'}^{r'+t'} \left(\int_{-\infty}^{t'-r'} A^p \left(\frac{s-\tau}{2}\right)^{\beta-2} d\tau \right) ds\\
 & \lesssim_p A^p \int_{t'-r'}^{r'+t'} (s-t'+r')^{\beta-1} ds \lesssim_p A^p (r')^\beta.
\end{align*}

 \begin{figure}[h]
 \centering
 \includegraphics[scale=1.1]{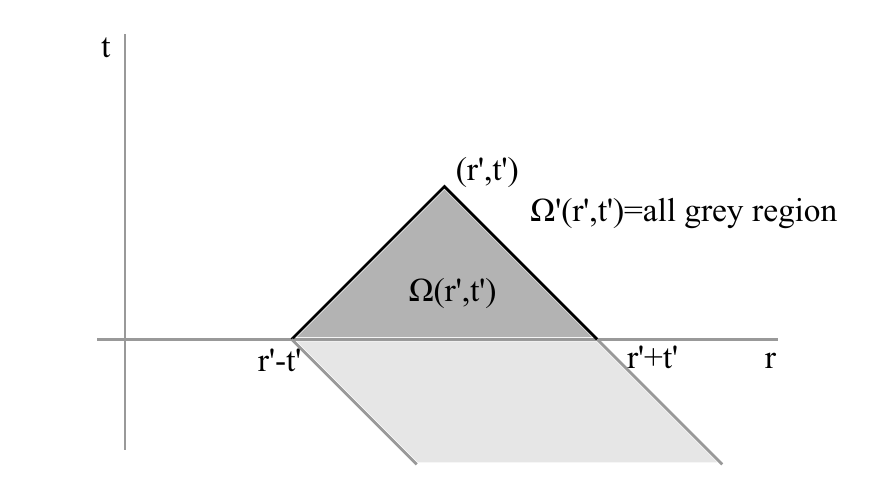}
 \caption{Illustration of proof for Lemma \ref{estimate of nonlinear Cbeta}} \label{figure triangle integral}
\end{figure}

\end{proof}
\noindent This helps us prove the following 
\begin{proposition} \label{pointwise decay in Cbeta}
 Let $w$ be a solution to the one-dimensional wave equation 
 \[
  \left\{\begin{array}{ll}\partial_t^2 w - \partial_r^2 w = - \frac{|w|^{p-1}w}{r^{p-1}}, & (r,t) \in \Rm^+ \times \Rm;\\ \displaystyle \lim_{r\rightarrow 0^+} w(r,t) = 0; & \\ (w,w_t)|_{t=0} = (w_0,w_1); &  \end{array} \right.
 \]
 with $(w_r(\cdot,t),w_t(\cdot,t)) \in C(\Rm; L^2(\Rm^+) \times L^2(\Rm^+))$. In addition, the initial data $(w_0,w_1)$ satisfy
 \begin{align*}
   &|w_0(r)| \leq c r^\beta, \quad r\geq R;&
   &\left|\int_{r_1}^{r_2} w_1(s) ds \right| \leq c (r_2-r_1)^{\beta},\quad R\leq r_1<r_2.&
 \end{align*}
 If $c<c(p)$ is sufficiently small, then the inequality $w(r,t) < 3c r^{\beta}$ holds for all $r\geq |t|+R$.
\end{proposition}
\begin{proof}
 Since the wave equation is time-reversible, we only need to prove the inequality in the positive time direction $t\geq 0$. Let us consider the set
 \[
  S_+ = \{(r,t): t\geq 0, r\geq t+R, |w(r,t)| \geq 3c r^{\beta}\}.
 \]
 Our goal is to show this set $S_+$ is empty. Since $(w_r(\cdot,t),w_t(\cdot,t)) \in C(\Rm; L^2 \times L^2)$ and $w(0^+,t)=0$, it immediately follows that $w$ is a continuous function of $r,t$. Thus the set $S_+$ is closed. If $S_+$ were nonempty, there would exist $(r_0,t_0) \in S_+$ so that $(r,t) \notin S_+$ for all $(r,t)$ with $r+t<r_0+t_0$. Since we have already known $|w(r,0)|=|w_0(r)| \leq c r^\beta$ when $t=0$, the time $t_0$ must be positive. By d'Alembert formula, we have 
\[
 w(r_0,t_0) = \frac{1}{2}[w_0(r_0-t_0)+w_0(r_0+t_0)]+\frac{1}{2}\int_{r_0-t_0}^{r_0+t_0} w_1(r) dr - \frac{1}{2} \int_{\Omega(r_0,t_0)} \frac{|w|^{p-1}w(r,t)}{r^{p-1}} drdt.
\]
Here the region $\Omega(r_0,t_0)$ is defined in the same way as in the previous lemma. All interior points $(r,t)$ of $\Omega(r_0,t_0)$ satisfy $r+t<r_0+t_0$. This means that none of these points is contained in the set $S_+$. In other words, All interior points $(r,t)$ satisfy $|w(r,t)| < 3c r^\beta$. As a result, we are able to apply Lemma \ref{estimate of nonlinear Cbeta} on the double integral above. We can also plug in the upper bound of initial data at the same time and obtain
\[
 |w(r_0,t_0)|  \leq \frac{1}{2}\left[c(r_0-t_0)^\beta + c(r_0+t_0)^\beta + c(2t_0)^\beta\right]+ C_p (3c)^p r_0^\beta \leq 2c r_0^\beta + 3^p C_p c^p r_0^\beta
\]
Since we assumed $|w(r_0,t_0)| \geq 3c r_0^\beta$, we have 
\[
 3c r_0^\beta \leq 2c t_0^\beta + 3^p C_p c^p r_0^\beta \; \Rightarrow \; c \leq 3^p C_p c^p.
\]
This gives a contradiction if $c<c(p)$ is sufficiently small.
\end{proof}
\paragraph{An explicit example} Now let us choose radial initial data of (CP1) in the following way
\begin{align*}
 &u_0 \in C^\infty(\Rm^3),\quad u_0(x) = c|x|^{-\frac{2}{p-1}}, \; |x|\geq 1;& &u_1 = 0.&
\end{align*}
Here $c$ is a small constant. A basic calculation shows that the energy of $(u_0,u_1)$ decays like 
\[
 \frac{1}{2}|\nabla u_0(x)|^2 + \frac{1}{2}|u_1(x)|^2 + \frac{1}{p+1}|u_0(x)|^{p+1} \lesssim |x|^{-2-\frac{4}{p-1}}, \quad |x|\gg 1.
\]
Thus for a positive constant $\kappa <\frac{5-p}{p-1}$ we have
\[
 \int_{\Rm^3} (1+|x|^\kappa) \left(\frac{1}{2}|\nabla u_0(x)|^2 + \frac{1}{2}|u_1(x)|^2 + \frac{1}{p+1}|u_0(x)|^{p+1}\right) dx < \infty.
\]
We may choose $\kappa \in (\frac{5-p}{p+1}, \frac{5-p}{p-1})$, apply Theorem \ref{scattering with full energy decay} and conclude that the corresponding solution $u$ to (CP1) with initial data $(u_0,u_1)$ scatters with $\|u\|_{L^p L^{2p}(\Rm^+\times \Rm^3)}<+\infty$. On the other hand, Proposition \ref{pointwise decay in Cbeta} gives an estimate on the function $w = ru$
\begin{equation} \label{w upper bound app}
 |w(r,t)| \leq 3c r^\beta, \qquad r\geq 1+t.
\end{equation}
For all $(r,t)$ with $r\geq 1+t$ we may apply d'Alembert formula and obtain
\begin{align*}
 w(r,t) & = \frac{1}{2}[w(r-t,0)+w(r+t,0)] + 0 + \frac{1}{2} \int_{\Omega(r,t)} \frac{-|w|^{p-1}w(r',t')}{(r')^{p-1}} dr' dt'\\
 & = \frac{1}{2}c[(r-t)^\beta+(r+t)^\beta] + \frac{1}{2} \int_{\Omega(r,t)} \frac{-|w|^{p-1}w(r',t')}{(r')^{p-1}} dr' dt'.
\end{align*}
By Lemma \ref{estimate of nonlinear Cbeta} and our estimate \eqref{w upper bound app}, we have 
\[
 \left|\int_{\Omega(r,t)} \frac{-|w|^{p-1}w(r',t')}{(r')^{p-1}} dr' dt' \right| \lesssim_p c^p r^\beta.
\]
Therefore if $c$ is sufficiently small, we always have ($r>1+t$)
\[
 w(r,t) \simeq c r^{\beta} \quad \Rightarrow \quad u(r,t) \simeq c r^{-\frac{2}{p-1}}.
\]
A simple calculation then shows
\[
 \int_0^\infty \int_{|x|>1+t} |u(x,t)|^{2(p-1)} dx dt = +\infty\quad \Rightarrow \quad \|u\|_{L^{2(p-1)}L^{2(p-1)}(\Rm^+ \times \Rm^3)} = + \infty.
\] 

\end{document}